\documentclass[a4paper,12pt]{scrartcl}

\usepackage[utf8]{inputenc}
\usepackage[english]{babel}
\usepackage{amssymb}
\usepackage{amsmath}
\usepackage{latexsym}
\usepackage{amsthm}
\usepackage{eucal}
\usepackage{graphicx}
\usepackage{bbm}
\usepackage{verbatim}
\usepackage{epigraph}
\usepackage{mathtools}
\usepackage{authblk}
\usepackage[pdflatex]{crop}
\usepackage[bookmarks]{hyperref}
\usepackage{tikz}
\usepackage{tikz-cd}
\usepackage{microtype}
\usepackage{enumitem}
\usetikzlibrary{babel,decorations.markings}
\allowdisplaybreaks

\setkomafont{disposition}{\normalfont\bfseries}
\newcommand{\auth}[0]{{Tobias Fritz and Paolo Perrone}}
\newcommand{\tit}[0]{{A Criterion for Kan Extensions of Lax Monoidal Functors}}
\newcommand{\kw}[0]{{Kan extension, lax monoidal functor, pseudomonad, pseudoalgebra, lax morphism}}

\hypersetup{
pdfauthor={\auth},%
pdftitle={\tit},%
colorlinks, linktocpage=true, pdfstartpage=1, pdfstartview=FitV,%
breaklinks=true, pdfpagemode=UseNone, pageanchor=true, pdfpagemode=UseOutlines,%
plainpages=false, bookmarksnumbered, bookmarksopen=true, bookmarksopenlevel=1,%
hypertexnames=true, pdfhighlight=/O,%
urlcolor=black, linkcolor=black, citecolor=black, 
}
\numberwithin{equation}{section}

\theoremstyle{plain}
\newtheorem{thm}{Theorem}[section]

\theoremstyle{definition}
\newtheorem{remark}[thm]{Remark}

\newcommand{\Z}{\mathbb{Z}}
\newcommand{\N}{\mathbb{N}}

\newcommand{\cat}[1]{{\mathsf{#1}}} 
\newcommand{\ar}[2][]{\arrow{#2}{#1}}

\newcommand{\id}{\mathrm{id}} 
\newcommand{\List}{\mathrm{List}}

\let\originalleft\left
\let\originalright\right
\renewcommand{\left}{\mathopen{}\mathclose\bgroup\originalleft}
\renewcommand{\right}{\aftergroup\egroup\originalright}

\tikzset{shorten <>/.style={shorten >=#1,shorten <=#1}}

\title{\tit}
\author{\auth}
\affil{Max Planck Institute for Mathematics in the Sciences,\\ Leipzig, Germany}
\date{}

\begin{document}

\maketitle

\begin{abstract}
In this mainly expository note, we state a criterion for when a left Kan extension of a lax monoidal functor along a strong monoidal functor can itself be equipped with a lax monoidal structure, in a way that results in a left Kan extension in $\cat{MonCat}$. This belongs to the general theory of algebraic Kan extensions, as developed by Melli\`es--Tabareau, Koudenburg and Weber, and is very close to an instance of a theorem of Koudenburg. We find this special case particularly important due to its connections with the theory of graded monads.
\end{abstract}

\section{Introduction}

An \emph{algebraic Kan extension}~\cite{algkanweber, koudenburg} is a Kan extension in the 2-category of algebras of a 2-monad. Intuitively, it is a Kan extension which preserves a given algebraic structure. An important example is a \emph{monoidal Kan extension}, which is a Kan extension in the 2-category of monoidal categories, lax monoidal functors, and monoidal natural transformations. 

Consider now the following problem: suppose that we have a diagram in $\cat{Cat}$
\begin{equation}
\label{ext}
  \begin{tikzcd}
   \cat{C} \ar[swap]{dr}{G} \ar[""{name=F,below}]{rr}{F} 
    \ar[bend right=75, ""{name=GL,above},phantom]{rr} 
     \ar[Rightarrow,from=F,to=GL,"\lambda"]
    && \cat{D} \\
   & \cat{C'} \ar[swap]{ur}{L}
  \end{tikzcd}
 \end{equation}
exhibiting $(L,\lambda)$ as the left Kan extension of $F$ along $G$, meaning that for every $X : C' \to D$, composition with $\lambda$ gives a bijection
\[
	\cat{Cat}(L,X) \cong \cat{Cat}(F,XG).
\]
Suppose moreover that $\cat{C},\cat{C'}$ and $\cat{D}$ are monoidal categories, and that $F$ and $G$ are lax monoidal functors. Under which conditions can we equip $L$ with a canonical lax monoidal structure inherited from those of $F$ and $G$? And under which conditions does this make $(L,\lambda)$ a Kan extension in the 2-category of monoidal categories and monoidal functors?

Problems of this kind have been considered several times in the literature under various assumptions (Section~\ref{relatedwork}). In some of the approaches, including ours, a relevant requirement is that $G$ should be \emph{strong monoidal}.\footnote{A notable exception is Koudenburg's approach~\cite{koudenburg}, where the usage of double categories allows for $G$ to be merely oplax monoidal.}
An intuitive reason is the following: given two objects $X$ and $Y$ in $\cat{C'}$, a monoidal structure of the functor $L$ should give a map $LX\otimes LY\to L(X\otimes Y)$. The way to obtain such a map from the monoidal structures of $F$ and $G$, together with the universal property, is to apply the multiplication \emph{backward} along $G$ in some sense, and then forward along $F$; see the upcoming~\eqref{cylinderup} and~\eqref{cylinderdown}. 
Another relevant condition is that the monoidal structure of $\cat{D}$ should be compatible with Kan extensions, in the sense that if $(L,\lambda)$ is the Kan extension of $F$ along $G$, then $(L\otimes L,\lambda\otimes\lambda)$ is the Kan extension of $F\otimes F$ along $G\times G$ (see the statement of Theorem~\ref{monoidalkanext} for the rigorous formulation). 
Given these two requirements, we show in Section~\ref{moncat} that $L$ admits a unique lax monoidal structure which turns $\lambda$ into a monoidal natural transformation, and that $(L,\lambda)$ is a Kan extension in the 2-category of monoidal categories and monoidal functors.

\subsection{Application: monads from graded monads}
\label{sec_appl}

Many monads that come up in mathematics and computer science are \emph{graded}, in the sense that they arise from gluing together a family of functors via a colimit. For example, the \emph{free monoid monad} or \emph{list monad} $\mathrm{List}:\cat{Set}\to\cat{Set}$, which assigns to every set $A$ the collection of finite sequences of elements of $A$, is naturally graded by the length of the list,
\begin{equation}\label{list}
	\List(A) = \coprod_{n\in\N} A^n,
\end{equation}
where the cartesian power $A^n$ is the set of lists of length $n$. So if we define $\List_n(A) := A^n$, then we can consider this structure as a lax monoidal functor $\List : \N\to [\cat{Set},\cat{Set}]$, where $\N$ is the monoid of natural numbers under multiplication, considered as a monoidal category with only identity morphisms, and $[\cat{Set},\cat{Set}]$ is the (not locally small) category of endofunctors\footnote{One can avoid this size problem by uncurrying to $\N\times\cat{Set}\to\cat{Set}$~\cite{ssm}, but then our change-of-grading considerations in terms of Kan extensions would no longer be covered by our Theorem~\ref{monoidalkanext}.}. An appealing feature of this perspective is that the multiplication
\[
	\List_m(\List_n(A)) \longrightarrow \List_{mn}(A),
\]
which unfolds an $m\times n$ matrix into a list of length $mn$, is now actually an isomorphism, corresponding to the currying isomorphism $(A^n)^m\cong A^{mn}$. The original list functor~\eqref{list} can be obtained from this by noting that it is the colimit of $\List_n$ over $n$, or equivalently the (pointwise) left Kan extension of $\List:\N\to[\cat{Set},\cat{Set}]$ along the unique strong monoidal functor $\N\to\cat{1}$. But does this construction also recover the monad structure? This is the type of question that we address here, motivated by our investigations into categorical probability theory, where a situation quite similar to $\List$ has come up~\cite{ours_kantorovich}. We will see that the answer is positive in many cases: the complete monad structure can be recovered under mild assumptions, and even results in a Kan extension in $\cat{MonCat}$, the 2-category of monoidal categories, lax monoidal functors, and monoidal transformations. This is the type of situation in which our Theorem~\ref{monoidalkanext} can be applied: it provides a criterion for which the pieces of structure that make up the graded monad can be put together to an ungraded monad.

As a similar example indicating that a positive answer should be expected, let us consider rings and modules graded by a monoid $M$ with neutral element $1$. An $\emph{$M$-graded ring}$ is a family of abelian groups $(R_m)_{m\in M}$ together with a unit element $1\in R_1$ and a family of multiplication maps
\[
	R_m\otimes R_n\longrightarrow R_{mn},
\]
which are homomorphisms of abelian groups and satisfy the usual associativity and unit laws. More concisely, an $M$-graded ring is a lax monoidal functor $M\to\cat{Ab}$. Now taking the direct sums $\bigoplus_m R_m$ turns this graded ring into an ungraded ring. In fact, a very common abuse of notation is to write $\bigoplus_m R_m$ \emph{itself} for the \emph{graded} ring. It is clear that as a functor $\bigoplus_m R_m : 1\to\cat{Ab}$, the underlying abelian group of this ungraded ring is the left Kan extension of $R : M\to\cat{Ab}$ along $M\to\cat{1}$. And indeed also the ring structure of this extension arises from our Theorem~\ref{monoidalkanext}, characterizing the ring $\bigoplus_m R_m$ as a Kan extension in $\cat{MonCat}$. Similar comments apply to other changes of grading $M\to M'$ for nontrivial $M'$.

\subsection{Related work}\label{relatedwork}

$\cat{MonCat}$ is the 2-category of pseudoalgebras of the free monoidal category 2-monad, with lax algebra morphisms and algebraic transformations. So a Kan extension in $\cat{MonCat}$ which is also a Kan extension in $\cat{Cat}$ is an instance of an \emph{algebraic} Kan extension, i.e.~a Kan extension in a 2-category of algebras of a 2-monad $T$ which interacts well with the $T$-algebra structure.

There are many possible variants of statements on algebraic Kan extensions. Does one work with particular 2-monads on particular 2-categories, or more general situations? In the second case, does one assume a strict 2-monad or some kind of weak 2-monad? Pseudoalgebras or lax algebras? Pseudomorphisms, lax, or oplax morphisms? Is the lax or pseudomorphism structure on the extension assumed to be given, or is it constructed in terms of the other data? We now summarize the existing results that we are aware of.

The apparently first results in this direction were obtained by Getzler~\cite{getzler}, who investigated operads and related universal-algebraic structures from the perspective of symmetric monoidal categories. His Proposition 2.3 provides a criterion for when a left Kan extension~\eqref{ext} in $\mathcal{V}$-$\cat{Cat}$ lifts to a left Kan extension of strong symmetric monoidal $\mathcal{V}$-functors between symmetric monoidal $\mathcal{V}$-categories.

Shortly after appears the apparently independent work of Melli\'es and Tabareau~\cite{mt}. They consider a pseudomonad $T$ on a proarrow equipment, and think of it as a framework for categorical algebra, where $T$-algebras play the role of generalized algebraic theories. Their Theorem 1 provides a condition for when a left Kan extension lifts to the 2-category of pseudoalgebras and pseudomorphisms, and this provides insight into when a $T$-algebraic theory has free models.

Paterson~\cite{paterson} has investigated general constructions of lax monoidal functors for functional programming. His Proposition 4 provides a construction of a lax monoidal structure on a left Kan extension~\eqref{ext}, under certain assumptions including that $\cat{D}$ is cartesian monoidal and that the Kan extension is pointwise.

Koudenburg has investigated algebraic Kan extensions as the main theme of~\cite{koudenburg}, working in the setting of double monads on double categories. Also his follow-up work~\cite{hypervirtual} touches upon this subject in Section 7, in the related but even more general setting of hypervirtual double categories. The main results of~\cite{koudenburg}, its Theorem 1.1 and the more general Theorem 5.7, substantially generalize Getzler's theorem, with Theorem~5.7 even considering lax algebras. As we will explain in Remark~\ref{rel_to_koudenburg}, our Theorem~\ref{monoidalkanext} is very close to a special case of Theorem 7.7 of~\cite{hypervirtual}.

Weber~\cite{algkanweber} has also investigated algebraic Kan extensions, motivated by phenomena arising in operad theory, revolving around a notion of \emph{exact square}. His Theorem 2.4.4 builds a lax structure on a Kan extension using a construction similar to the first part of our proof. Part (c) of this theorem then gives a criterion for when this construction results in a left Kan extension of pseudomorphisms, based on his notion of exactness; see Remark~\ref{rel_to_weber}.

Finally, we have also addressed a similar problem for the case of monoidal categories in Appendix B of~\cite{ours_kantorovich}. Our Theorem~B.1 in there is similar to the present Theorem~\ref{monoidalkanext}, but with different assumptions: the result in \cite{ours_kantorovich} requires $L$ already to carry a lax monoidal structure, while the assumptions on the left Kan extension $\lambda$ were relaxed: the natural transformation $\lambda\otimes\lambda$ only needed to be a vertical epimorphism, while we now require even $\lambda\otimes\lambda\otimes\lambda$ to be a Kan extension as well.

In conclusion, due to the overlap with previous work and in particular Koudenburg's results, the present manuscript should not be considered original research, but rather as a convenient reference for Kan extensions of lax monoidal functors, stating and proving an essentially known criterion in a somewhat more pedestrian way than the general 2-monadic theory developed by Melli\`es--Tabareau, Koudenburg, and Weber, who deserve the essential credit.

\section{Kan extensions of lax monoidal functors}\label{moncat}

We first fix some notation. For a monoidal category $\cat{C}$, we denote its structure 1-morphisms as
\begin{equation*}
  e: 1\to \cat{C}, \qquad \otimes:\cat{C}\times\cat{C} \to \cat{C},
\end{equation*}
using a subscript $\cat{C}$ only when necessary for disambiguation. For a lax monoidal functor $F:\cat{C}\to\cat{D}$, we denote its multiplication by $\mu_F : F(-)\otimes F(-)\rightarrow F(-\otimes -)$ and its unit by $\eta_F : 1\to F(1)$. $\cat{MonCat}$ is the 2-category of monoidal categories, lax monoidal functors, and monoidal natural transformations.

For functors $F,\, G : \cat{C}\to\cat{D}$ with $\cat{D}$ monoidal and $\alpha : F \Rightarrow G$ a natural transformation, we also write $\alpha\otimes\alpha$ for the whiskered transformation
\[\begin{tikzcd}
   \cat{C}\times\cat{C} \ar[swap,bend right=40,""{name=G,above}]{rr}{G\times G} \ar[bend left=40,""{name=F,below}]{rr}{F\times F} 
     \ar[Rightarrow,from=F,to=G,"\alpha\times\alpha"]
    && \cat{D}\times\cat{D} \ar{r}{\otimes} & \cat{D} 
\end{tikzcd}\]
which goes from $F\otimes F$ to $G\otimes G$.

\begin{thm}
\label{monoidalkanext}
Suppose that 
\begin{enumerate}
 \item\label{cats_assp} $\cat{C}$, $\cat{C'}$ and $\cat{D}$ are monoidal categories;
 \item $F:\cat{C}\to\cat{D}$ is lax monoidal and $G:\cat{C}\to\cat{C'}$ is strong monoidal;
 \item\label{kan_ext_assp}
 \begin{equation*}
  \begin{tikzcd}
   \cat{C} \ar[swap]{dr}{G} \ar[""{name=F,below}]{rr}{F} 
    \ar[bend right=75, ""{name=GL,above},phantom]{rr} 
     \ar[Rightarrow,from=F,to=GL,"\lambda"]
    && \cat{D} \\
   & \cat{C'} \ar[swap]{ur}{L}
  \end{tikzcd}
 \end{equation*}
 is a left Kan extension in $\cat{Cat}$;
 \item\label{ourpreserve} $\lambda\otimes\lambda$ also makes $L\otimes L$ into the left Kan extension of $F\otimes F$ along $G\times G$,
 \begin{equation*}
  \begin{tikzcd}
   \cat{C}\times\cat{C} \ar[swap]{dr}{G\times G} \ar[""{name=F,below}]{rr}{F\times F} 
    \ar[bend right=50, ""{name=GL,above},phantom]{rr} 
     \ar[Rightarrow,from=F,to=GL,"\lambda\times\lambda"]
    && \cat{D}\times\cat{D} \ar{r}{\otimes} & \cat{D} \\
   & \cat{C'}\times\cat{C'} \ar[swap]{ur}{L\times L}
  \end{tikzcd}
 \end{equation*}
 and similarly for $\lambda\otimes\lambda\otimes\lambda$.
\end{enumerate}
Then we also have that
\begin{enumerate}[resume]
 \item\label{unique_lax_claim}There is a unique lax monoidal structure on $L$ which makes $\lambda$ into a monoidal natural transformation, meaning that it is compatible with the units,
 \begin{equation}\label{unitlambda}
  \begin{tikzcd}
   & \cat{C} \ar[bend left]{dr}{G} \\
   \cat{1} \ar{ur}{e} \ar[near start]{rr}{e} \ar[swap]{dr}{e} 
    \ar[bend left=10,""{name=ECP},phantom]{rr} \ar[bend left=60,""{name=ECG},phantom]{rr}
     \ar[Rightarrow,bend right=20,from=ECP,to=ECG,"\eta_G",swap]
    \ar[bend left=10,""{name=ED},phantom]{dr} \ar[bend left=90,""{name=ECPL,pos=0.6},phantom]{dr}
     \ar[Rightarrow,bend right=20,from=ED,to=ECPL,"\eta_L",swap]
    && \cat{C'} \ar[bend left]{dl}{L} \\
   & \cat{D} 
  \end{tikzcd}
  \quad \equiv \quad
  \begin{tikzcd}
   & \cat{C} \ar[bend left]{dr}{G} \ar[bend right=20]{dd}{F}
    \ar[bend left=10,""{name=F},phantom]{dd} \ar[bend left=90,""{name=GL},phantom]{dd}
     \ar[Rightarrow,from=F,to=GL,"\lambda"] \\
   \cat{1} \ar{ur}{e} \ar[""{name=un, above},swap]{dr}{e} 
    \ar[bend left=90,""{name=ECF,pos=0.4},phantom]{dr} \ar[bend left=10,""{name=ED},phantom]{dr}
     \ar[Rightarrow,bend left=20,from=ED,to=ECF,"\eta_F"]
    && \cat{C'} \ar[bend left]{dl}{L} \\
   &\cat{D} 
  \end{tikzcd}
 \end{equation} 
 and with the multiplications, 
 \begin{equation}\label{multlambda}
  \begin{tikzcd}[column sep=small]
    \cat{C}\times\cat{C} \ar[swap, bend right=50]{dd}{F\times F} 
    \ar[swap, bend right=30,""{name=TF},phantom]{dd} \ar[swap, bend left=55,""{name=TGTL},phantom]{dd}
     \ar[Rightarrow,from=TF,to=TGTL,"\lambda\times\lambda"]
    \ar{rr}{\otimes} \ar[bend left,swap]{dr}{G\times G} 
    \ar[bend left=30,""{name=AG,below,pos=0.65},phantom]{drrr} \ar[bend right=10,""{name=TGA},phantom]{drrr}
     \ar[Rightarrow,from=TGA,to=AG,bend right=20,"\mu_G"]
   && \cat{C} \ar[bend left]{dr}{G} \\
   &  \cat{C'}\times\cat{C'} \ar[bend left,swap]{dl}{L\times L} \ar{rr}{\otimes}
    \ar[bend left=50,""{name=AL},phantom]{dr} \ar[bend right=50,""{name=TLA},phantom]{dr}
     \ar[Rightarrow, bend right=20, from= TLA, to=AL,"\mu_L"]
    && \cat{C'} \ar[bend left]{dl}{L} \\
   \cat{D}\times\cat{D} \ar[swap]{rr}{\otimes} && \cat{D}   
  \end{tikzcd}
  \; \equiv \;
  \begin{tikzcd}
    \cat{C}\times\cat{C} \ar[swap, bend right=50]{dd}{F\times F} \ar{rr}{\otimes}
    \ar[bend right=75,""{name=TFA,pos=0.45},phantom]{ddrr} \ar[bend left=50,""{name=AF,pos=0.4},phantom]{ddrr}
     \ar[Rightarrow,from=TFA,to=AF,"\mu_F",bend left=30]
   && \cat{C} \ar[swap,bend right=50]{dd}{F} \ar[bend left]{dr}{G}
    \ar[bend right=30,""{name=F},phantom]{dd} \ar[bend left=90,""{name=GL},phantom]{dd}
     \ar[Rightarrow, from=F,to=GL,"\lambda"] \\
    &&& \cat{C'} \ar[bend left]{dl}{L} \\
   \cat{D}\times\cat{D} \ar[swap]{rr}{\otimes} && \cat{D}
  \end{tikzcd}
 \end{equation}
 \item\label{lax_kan_claim} With this structure, $\lambda$ also makes $L$ the left Kan extension of $F$ along $G$ in $\cat{MonCat}$.
\end{enumerate}
\end{thm}

It may help to visualize these and the equations used in the proof below three-dimensionally, by interpreting every rewriting step as a globular 3-cell, and whiskering and composing these 3-cells so as to form a 3-dimensional pasting diagram. Like this,~\eqref{multlambda} becomes a full cylinder, with the two caps formed by $\lambda\times\lambda$ and $\lambda$, and with the three multiplications wrapping around. The same is true for equation~\eqref{unitlambda}, but with the $\lambda\times\lambda$ cap collapsed to a single point, so that one obtains a cone with $\lambda$ on the base.

\tikzstyle{barred} = [decoration={markings, mark=at position 0.5 with {\draw[-] (0,-1.5pt) -- (0,1.5pt);}},postaction={decorate}]	

\begin{remark}
\label{rel_to_koudenburg}
As we explain now, this theorem is very similar to an instance of Koudenburg's Theorem~7.7(a) in~\cite{hypervirtual}, and is also closely related to the horizontal dual of his Theorem~5.7(a) in~\cite{koudenburg}. In this remark, all theorem numbers refer to~\cite{hypervirtual}, unless stated otherwise.

We use the hypervirtual double category $\cat{(Set,Set')\text{-}Prof}$ of Example~1.8. The Kan extensions that we work with are \emph{weak Kan extensions}, as per Definition~4.1, since by Proposition~2.18, Kan extensions in the vertical 2-category $\cat{Cat}$ along a functor $G$ are equivalently weak Kan extensions along the companion $G_*$.

We consider the free strict monoidal category double monad $T(-) = \coprod_{n\in\N} (-)^{\times n}$ (Example~6.2). Thus by Theorem~7.7(a), we get the following: if we have categories, functors, and a left Kan extension $\lambda$ as in \ref{cats_assp}--\ref{kan_ext_assp} of our assumptions above, and if moreover
 \begin{equation}
 \label{Tpreserve}
  \begin{tikzcd}
   T\cat{C} \ar[swap]{dr}{TG} \ar[""{name=F,below}]{rr}{TF} 
    \ar[bend right=50, ""{name=GL,above},phantom]{rr} 
     \ar[Rightarrow,from=F,to=GL,"T\lambda"]
    && T\cat{D} \ar{r}{\otimes} & \cat{D} \\
   & T\cat{C'} \ar[swap]{ur}{TL}
  \end{tikzcd}
 \end{equation}
is a Kan extension in $\cat{Cat}$, and similarly
 \begin{equation}
 \label{TTpreserve}
  \begin{tikzcd}
   TT\cat{C} \ar[swap]{dr}{TTG} \ar[""{name=F,below}]{rr}{TTF} 
    \ar[bend right=50, ""{name=GL,above},phantom]{rr} 
     \ar[Rightarrow,from=F,to=GL,"TT\lambda"]
    && TT\cat{D} \ar{r}{T\otimes} & T\cat{D} \ar{r}{\otimes} & \cat{D} \\
   & TT\cat{C'} \ar[swap]{ur}{TTL}
  \end{tikzcd}
 \end{equation}
is a Kan extension in $\cat{Cat}$, then our claims \ref{unique_lax_claim} and~\ref{lax_kan_claim} follow. Now the main assumptions \eqref{Tpreserve}--\eqref{TTpreserve} are clearly very similar to our assumption~\ref{ourpreserve}, and we think of them as the unbiased version of~\ref{ourpreserve}. However, technically we do not know how the two assumptions are precisely related, although Example~5.3 of~\cite{koudenburg} gives a partial result. It is not even clear whether the fact that $\lambda\otimes\lambda$ and $\lambda\otimes\lambda\otimes\lambda$ are Kan extensions implies that arbitrary tensor powers of $\lambda$ are Kan extension as well, since the additional functor $X : T\cat{C'} \to \cat{D}$ that one needs to consider may not interact well with products.

Nevertheless, due to the similarity of the assumptions, we regard our Theorem~\ref{monoidalkanext} as essentially an instance of Koudenburg's work. In particular, the proofs of the two theorems are very similar as well, and we expect there to be a common generalization that we are not aware of.

We therefore think that the essential credit for our Theorem~\ref{monoidalkanext} should go to Koudenburg, as well as to the other authors who have worked on the general theory of algebraic Kan extensions.
\end{remark}

\newcommand{\initialmorph}{\text{!`}}	

\begin{remark}
\label{rel_to_weber}
There are also some close similarities with the work of Weber, and in particular his Theorem 2.4.4 from~\cite{algkanweber}. However, technically he imposes an exactness condition which plays a central role in his approach, and which does not hold in many of our motivating examples described in Section~\ref{sec_appl}. Concretely, let $M := 1$ be the single-element monoid and $M' := \Z_2$ the group with two elements, and let us consider left Kan extension of lax monoidal functors along the unique monoid homomorphism $\initialmorph : 1 \to \Z_2$, corresponding to the change of grading which considers an ungraded algebra as a trivially $\Z_2$-graded algebra. With $T : \cat{Cat} \to \cat{Cat}$ the monoidal category 2-monad as above, Weber's approach would require the invertible 2-cell
\[
  \begin{tikzcd}
   T1 \ar{d}{\otimes} \ar[""{name=U,below}]{rr}{T\initialmorph} 
    && T\Z_2 \ar{d}{\otimes} \\
   1 \ar[swap,""{name=D,above}]{rr}{\initialmorph} && \Z_2
     \ar[Rightarrow,from=U,to=D,"\cong"]
  \end{tikzcd}
\]
to be exact, which means that composing with it should turn a left Kan extension along $\initialmorph$ into a left Kan extension along $T\initialmorph$. However, this is not the case, as one can see explicitly by noting that $T\initialmorph$ is equal to the inclusion $\N \to \N^2, \: m \mapsto (m,0)$, and left extending a trivially $\N$-graded algebra results in a trivially $\N^2$-graded algebra, while using $\otimes : T\Z_2 \to \Z_2$ to pull back the trivial $\Z_2$-grading results in an $\N^2$-graded algebra which is not generally concentrated in degree $0$. Alternatively, as pointed out to us by Mark Weber, one can also apply~\cite[Proposition~4.5.1]{algkanweber} in order to reach the same conclusion.
\end{remark}

\begin{proof}[Proof of Theorem~\ref{monoidalkanext}]
The following definitions of $\eta_L$ and $\mu_L$ make $\lambda$ monoidal by construction:
\begin{itemize}
 \item Consider the natural transformation $e_\cat{D}\Rightarrow L\circ e_\cat{C'}$ given by the composition
 \begin{equation}\label{coneup}
  \begin{tikzcd}[column sep=small]
   &&& \cat{C'} \ar[bend left=50,""{name=L,below,pos=0.1}]{dd}{L}  \\
   \cat{1} \ar[bend left,""{name=ECP,below,pos=0.4}]{urrr}{e} \ar[bend right,swap,""{name=ED,above}]{drrr}{e} \ar{rr}{e} && \ar[Rightarrow, bend left=15, from=ED,"\eta_F"] \cat{C} \ar[Rightarrow,bend left=10,swap,to=ECP,"\eta_G^{-1}"] \ar[bend left,swap]{ur}{G} \ar[swap,bend right,""{name=F,above}]{dr}{F} 
    \ar[Rightarrow,from=F,to=L,"\lambda",swap] \\
   &&& \cat{D}  
  \end{tikzcd}
 \end{equation}
 This has already the desired form
 \begin{equation}\label{conedown}
  \begin{tikzcd}[column sep=small]
    &&& |[alias=ECPL]| \cat{C'} \ar[bend left=50]{dd}{L}  \\
   \cat{1} \ar[bend left,""{name=ECP,below,pos=0.4}]{urrr}{e} \ar[bend right,swap,""{name=ED,above}]{drrr}{e}  
    \ar[bend right=20,swap,""{name=ED,above},phantom]{drrr} 
     \ar[Rightarrow,from=ED,to=ECPL,"\eta_L",swap] \\
    &&& \cat{D}
  \end{tikzcd}
 \end{equation}
 It is clear that this $\eta_L$ is the only possible choice for the unit map which satisfies \eqref{unitlambda}.
 \item Consider now the natural transformation $F\otimes F\Rightarrow L\circ  (G\otimes G)$ given by the composition
 \begin{equation}\label{cylinderup}
  \begin{tikzcd}[column sep=small]
   & \cat{C'}\times\cat{C'}  \ar{rr}{ \otimes} && \cat{C'} \ar[bend left=50,""{name=L,below,pos=0.1}]{dd}{L}  \\
   \cat{C}\times\cat{C} \ar[bend left]{ur}{G\times G} \ar[swap, bend right]{dr}{F\times F} \ar{rr}{ \otimes} 
    \ar[bend right=20,""{name=AG,pos=0.6},phantom]{urrr} \ar[bend left=20,""{name=TGA,pos=0.4},phantom]{urrr}
     \ar[Rightarrow, bend left=20, from= AG, to=TGA,"\mu_G^{-1}",swap]
   \ar[bend left=30,""{name=AF,below,pos=0.6},phantom]{drrr} \ar[bend right=15,""{name=TFA,pos=0.4},phantom]{drrr}
     \ar[Rightarrow,from=TFA,to=AF,bend left=20,"\mu_F",swap]
   && \cat{C} \ar[bend left,swap]{ur}{G} \ar[swap,bend right,""{name=F,above}]{dr}{F}
    \ar[Rightarrow,from=F,to=L,"\lambda",swap] \\
   & \cat{D}\times\cat{D} \ar[swap]{rr}{ \otimes} && \cat{D}
  \end{tikzcd}
 \end{equation}
 By the universal property of the Kan extension of $\lambda\otimes\lambda$, there exists a unique natural transformation $\mu_L:L\otimes L\Rightarrow L(-\otimes -)$ such that \eqref{cylinderup} is equal to
 \begin{equation}\label{cylinderdown}
  \begin{tikzcd}[column sep=small]
   & \cat{C'}\times\cat{C'} \ar[swap,bend left=50,""{name=LL,below,pos=0.1}]{dd}{L\times L} \ar{rr}{ \otimes}
    \ar[bend right=38,""{name=TLA,pos=0.7},phantom]{ddrr} \ar[bend left=68,""{name=AL,pos=0.5},phantom]{ddrr}
     \ar[Rightarrow,from=TLA,to=AL,"\mu_L",bend right=20]
   & \,& \cat{C'} \ar[bend left=50]{dd}{L}  \\
   \cat{C}\times\cat{C} \ar[bend left]{ur}{G\times G} \ar[swap, bend right,""{name=FF,above}]{dr}{F\times F} 
    \ar[Rightarrow,from=FF,to=LL,"\lambda\times\lambda",pos=0.6] \\
   & \cat{D}\times\cat{D} \ar[swap]{rr}{ \otimes} && \cat{D}
  \end{tikzcd}
 \end{equation}
 By the universality of $\lambda\otimes\lambda$, it is clear that this $\mu_L$ is the only possible choice for the multiplication which satisfies \eqref{multlambda}.
\end{itemize}

Next, we have to check that this actually defines a lax monoidal structure, corresponding to unitality and associativity.
\begin{itemize}
 \item The left unit condition says that the composition
 \begin{equation}\label{prismup}
  \begin{tikzcd}
   & \cat{1}\times\cat{D} \ar{rr}{e\times \id } && \cat{D}\times\cat{D} \ar{dd}{\otimes} \\
   \cat{1}\times\cat{C'} \ar{ur}{\id \times L} \ar[swap]{ddrr}{\cong} \ar{rr}{e\times \id } 
    \ar[bend left=10,""{name=IDLEID,pos=0.4},phantom]{urrr} \ar[bend right=10,""{name=EIDLL,pos=0.6},phantom]{urrr}
     \ar[Rightarrow,from=IDLEID,to=EIDLL,"\eta_L\times\id "]
    \ar[bend left=5,""{name=ISO},phantom]{ddrr} \ar[bend left=70,""{name=EIDT,pos=0.65},phantom]{ddrr}
     \ar[Rightarrow,from=EIDT,to=ISO,"\ell"]
    && \cat{C'}\times\cat{C'} \ar[near start]{ur}{L\times L} \ar{dd}{\otimes} 
     \ar[bend left=60,""{name=LLT},phantom]{dr} \ar[bend right=60,""{name=TL},phantom]{dr} 
      \ar[Rightarrow,from=LLT,to=TL,"\mu_L"] \\
   & \, && \cat{D} \\
   && \cat{C'} \ar[swap]{ur}{L}
  \end{tikzcd}
  \end{equation}
 has to be equal to the composition
  \begin{equation}\label{prismdown}
  \begin{tikzcd}
   & \cat{1}\times\cat{D} \ar[swap]{ddrr}{\cong} \ar{rr}{e\times \id }  
    \ar[bend left=5,""{name=ISO},phantom]{ddrr} \ar[bend left=70,""{name=EIDT,pos=0.6},phantom]{ddrr}
     \ar[Rightarrow,from=EIDT,to=ISO,"\ell"]
    && \cat{D}\times\cat{D} \ar{dd}{\otimes} \\
   \cat{1}\times\cat{C'} \ar{ur}{\id \times L} \ar[swap]{ddrr}{\cong}  && \, \\
   &&& \cat{D} \\
   && \cat{C'} \ar[swap]{ur}{L}
  \end{tikzcd}
 \end{equation}
 where $\ell$ denotes the left unitor (of both categories), and the empty space commutes trivially. It may help to visualize these two composites as forming the boundary of a triangular prism, with the two $\ell$'s making up the triangular faces. What we would like to do is to show that these two composites are equal by providing a sequence of rewrites which fills this prism with 3-cells represented by equations. But instead of doing this directly, we precompose \eqref{prismdown} with $G$ and $\lambda$,
 \begin{equation*}
  \begin{tikzcd}
   && \cat{1}\times\cat{D} \ar[swap]{ddrr}{\cong} \ar{rr}{e\times \id } 
    \ar[bend left=5,""{name=ISO},phantom]{ddrr} \ar[bend left=70,""{name=EIDT,pos=0.6},phantom]{ddrr}
     \ar[Rightarrow,from=EIDT,to=ISO,"\ell"]
    && \cat{D}\times\cat{D} \ar{dd}{\otimes}  \\
   \cat{1}\times\cat{C} \ar[swap]{ddrr}{\cong} \ar[bend left=12,""{name=F,below}]{urr}{\id \times F} \ar{r}{\id \times G} & \ar[Rightarrow,from=F,"\id  \times\lambda"] \cat{1}\times\cat{C'} \ar[swap]{ur}{\id \times L} \ar[swap]{ddrr}{\cong}  && \, \\
   &&&& \cat{D} \\
   && \cat{C} \ar[swap]{r}{G} & \cat{C'} \ar[swap]{ur}{L}
  \end{tikzcd}
  \end{equation*}
 and now start to apply rewrites. As the first rewrite, we note that this composite is trivially equal to
  
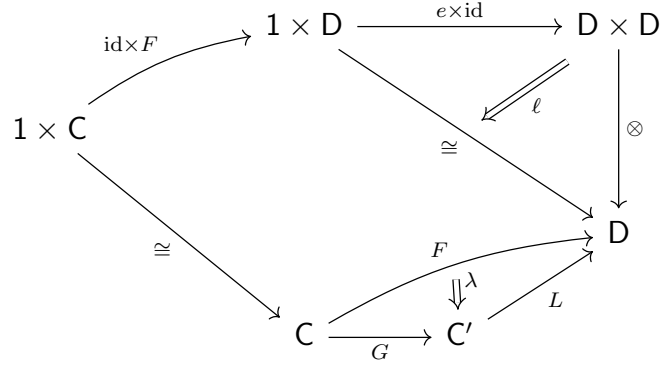
\begin{equation*}
  \begin{tikzcd}
   && \cat{1}\times\cat{D} \ar[swap]{ddrr}{\cong} \ar{rr}{e\times \id } 
    \ar[bend left=5,""{name=ISO},phantom]{ddrr} \ar[bend left=70,""{name=EIDT,pos=0.6},phantom]{ddrr}
     \ar[Rightarrow,from=EIDT,to=ISO,"\ell"] 
    && \cat{D}\times\cat{D} \ar{dd}{\otimes}  \\
   \cat{1}\times\cat{C} \ar[swap]{ddrr}{\cong} \ar[bend left=12,""{name=idF,below}]{urr}{\id \times F}  &&& \, \\
   &&&& \cat{D} \\
   && \cat{C} \ar[bend left=12,""{name=F,below}]{urr}{F} \ar[swap]{r}{G} & \ar[Rightarrow,from=F,"\lambda"] \cat{C'} \ar[swap]{ur}{L}
  \end{tikzcd}
 \end{equation*}
 By the left unit condition for $F$, this is furthermore equal to
 
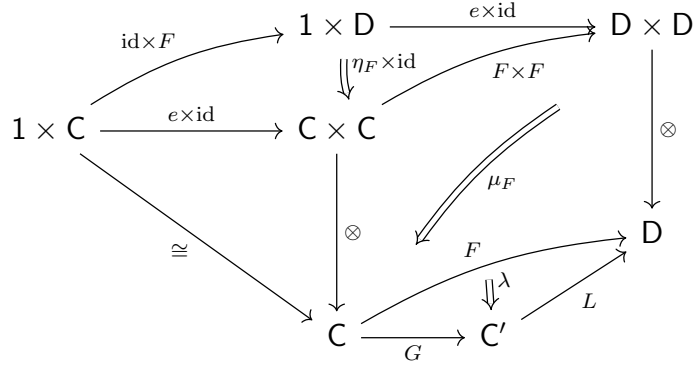
\begin{equation*}
  \begin{tikzcd}
   && \cat{1}\times\cat{D} \ar{rr}{e\times \id } && \cat{D}\times\cat{D} \ar{dd}{\otimes}  \\
   \cat{1}\times\cat{C} \ar[swap]{ddrr}{\cong} \ar[bend left=12,""{name=idF,below}]{urr}{\id \times F} \ar{rr}{e\times \id }  
    \ar[bend left=8,""{name=IDFEID},phantom]{urrrr} \ar[bend right=8,""{name=EIDFF},phantom]{urrrr} 
     \ar[Rightarrow,bend right=10,from=IDFEID,to=EIDFF,"\eta_F\times\id"]
    && \cat{C}\times\cat{C} \ar[bend left=12,swap]{urr}{F\times F} \ar{dd}{\otimes}  
    \ar[bend left=60,""{name=FFT,pos=0.6},phantom]{drr} \ar[bend right=60,""{name=TF,pos=0.35},phantom]{drr} 
     \ar[Rightarrow,bend right=10,from=FFT,to=TF,"\mu_F"]
    & \, \\
   & \, &&& \cat{D} \\
   && \cat{C} \ar[bend left=12,""{name=F,below}]{urr}{F} \ar[swap]{r}{G} & \ar[Rightarrow,from=F,"\lambda"] \cat{C'} \ar[swap]{ur}{L}
  \end{tikzcd}
 \end{equation*}
 which because of the monoidal structure of $\lambda$, as in~\eqref{multlambda}, is first equal to
 
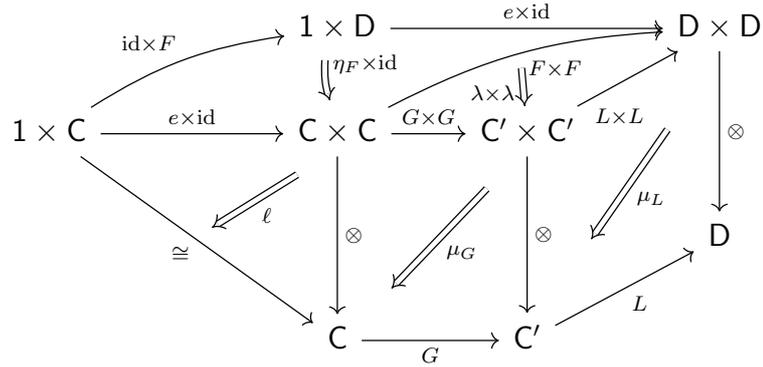
\begin{equation*}
  \begin{tikzcd}
   && \cat{1}\times\cat{D} \ar{rr}{e\times \id } && \cat{D}\times\cat{D} \ar{dd}{\otimes}  \\
   \cat{1}\times\cat{C} \ar[swap]{ddrr}{\cong} \ar[bend left=12,""{name=idF,below}]{urr}{\id \times F} \ar{rr}{e\times \id }  
    \ar[bend left=5,""{name=ISO},phantom]{ddrr} \ar[bend left=70,""{name=EIDT,pos=0.6},phantom]{ddrr}
     \ar[Rightarrow,from=EIDT,to=ISO,"\ell"]
    \ar[bend left=10,""{name=IDFEID,pos=0.4},phantom]{urrrr} \ar[bend right=5,""{name=EIDFF,pos=0.4},phantom]{urrrr} 
      \ar[Rightarrow,bend right=10,from=IDFEID,to=EIDFF,"\eta_F\times\id"]
    && \cat{C}\times\cat{C} \ar[bend left=12,swap,""{name=F,below}]{urr}{F\times F} \ar{dd}{\otimes}  \ar{r}{G\times G} 
     \ar[bend left=50,""{name=GGT,pos=0.45},phantom]{ddr} \ar[bend right=50,""{name=TG,pos=0.55},phantom]{ddr}
      \ar[Rightarrow, from=GGT, to=TG,"\mu_G"]
     & \ar[Rightarrow,from=F,"\lambda\times\lambda",swap,near start] \cat{C'}\times\cat{C'} \ar{dd}{\otimes} \ar[pos=0.1,swap]{ur}{L\times L} 
     \ar[bend left=60,""{name=LLT},phantom]{dr} \ar[bend right=60,""{name=TL},phantom]{dr} 
      \ar[Rightarrow,from=LLT,to=TL,"\mu_L"] \\
   & \, &&& \cat{D} \\
   && \cat{C}  \ar[swap]{r}{G} &  \cat{C'} \ar[swap]{ur}{L}
  \end{tikzcd}
 \end{equation*}
 and then, by \eqref{unitlambda}, can be written as
 \begin{equation*}
  \begin{tikzcd}
   && \cat{1}\times\cat{D} \ar{rr}{e\times \id }  && \cat{D}\times\cat{D} \ar{dd}{\otimes}  \\
   \cat{1}\times\cat{C} \ar[swap]{ddrr}{\cong} \ar[bend left=12,""{name=idF,below}]{urr}{\id \times F} \ar[swap]{rr}{e\times \id } \ar[bend left=20,""{name=EG,below}]{rrr}{e\times G} 
    \ar[bend left=8,""{name=IDFEID},phantom]{urrrr} \ar[bend right=8,""{name=EGLL,pos=0.7},phantom]{urrrr} 
     \ar[Rightarrow,bend left=10,from=IDFEID,to=EGLL,"\eta_L\times \lambda",pos=0.9]
    \ar[bend left=5,""{name=ISO},phantom]{ddrr} \ar[bend left=70,""{name=EIDT,pos=0.6},phantom]{ddrr}
     \ar[Rightarrow,from=EIDT,to=ISO,"\ell"]
    && \ar[Rightarrow,from=EG,"\eta_G\times \id ",swap,near start] \cat{C}\times\cat{C} \ar{dd}{\otimes}  \ar[swap]{r}{G\times G} 
     \ar[bend left=50,""{name=GGT,pos=0.45},phantom]{ddr} \ar[bend right=50,""{name=TG,pos=0.55},phantom]{ddr}
      \ar[Rightarrow, from=GGT, to=TG,"\mu_G"]
     &  \cat{C'}\times\cat{C'} \ar{dd}{\otimes} \ar{ur}{L\times L} 
     \ar[bend left=60,""{name=LLT},phantom]{dr} \ar[bend right=60,""{name=TL},phantom]{dr} 
      \ar[Rightarrow,from=LLT,to=TL,"\mu_L"] \\
   & \, &&& \cat{D} \\
   && \cat{C}  \ar[swap]{r}{G} &  \cat{C'} \ar[swap]{ur}{L}
  \end{tikzcd}
 \end{equation*}
 which in turn equals
 \begin{equation*}
  \begin{tikzcd}
   && \cat{1}\times\cat{D} \ar{rr}{e\times \id } && \cat{D}\times\cat{D} \ar{ddd}{\otimes} \\
   && |[alias=OC]| \cat{1}\times\cat{C'} \ar{u}{\id \times L} \ar[near start]{dr}{e\times \id }  
    \ar[bend left=30,""{name=IDLEID,pos=0.2},phantom]{urr} \ar[bend right=30,""{name=EIDLL,pos=0.4},phantom]{urr}
     \ar[Rightarrow, bend left=20,from=IDLEID,to=EIDLL,"\eta_L\times\id "] \\
   \cat{1}\times\cat{C} \ar[swap]{ddrr}{\cong} \ar[bend left=12]{uurr}{\id \times F} \ar[swap]{rr}{e\times \id } \ar{urr}{\id \times G} 
    \ar[bend left=5,""{name=IDF},phantom]{uurr} \ar[bend right=15,""{name=IDGIDL,pos=0.6},phantom]{uurr} 
     \ar[Rightarrow,from=IDF,to=IDGIDL,"\id\times\lambda"]
    \ar[bend left=30,""{name=IDGEID,pos=0.6},phantom]{rrr} \ar[bend left=10,""{name=EIDGG,pos=0.65},phantom]{rrr} 
      \ar[Rightarrow,from=IDGEID,to=EIDGG,"\eta_G\times\id",swap]
    \ar[bend left=5,""{name=ISO},phantom]{ddrr} \ar[bend left=70,""{name=EIDT,pos=0.6},phantom]{ddrr}
     \ar[Rightarrow,from=EIDT,to=ISO,"\ell"]
     &&  \cat{C}\times\cat{C} \ar{dd}{\otimes}  \ar[swap]{r}{G\times G} 
     \ar[bend left=50,""{name=GGT,pos=0.45},phantom]{ddr} \ar[bend right=50,""{name=TG,pos=0.55},phantom]{ddr}
      \ar[Rightarrow, from=GGT, to=TG,"\mu_G"]
     &  \cat{C'}\times\cat{C'} \ar{dd}{\otimes} \ar{uur}{L\times L} 
     \ar[bend left=90,""{name=LLT,pos=0.45},phantom]{dr} \ar[bend right=50,""{name=TL},phantom]{dr} 
      \ar[Rightarrow,from=LLT,to=TL,"\mu_L"] \\
   & \, &&& \cat{D} \\
   && \cat{C}  \ar[swap]{r}{G} &  \cat{C'} \ar[swap]{ur}{L}
  \end{tikzcd}
 \end{equation*}
 By the left unit condition for $G$, this is again equal to
 \begin{equation*}
 \begin{tikzcd}
   && \cat{1}\times\cat{D} \ar{rr}{e\times \id } && \cat{D}\times\cat{D} \ar{dd}{\otimes}  \\
   \cat{1}\times\cat{C} \ar[swap]{ddrr}{\cong} \ar[bend left=12,""{name=idF,below}]{urr}{\id \times F} \ar[swap]{r}{\id \times G} &  
   \ar[Rightarrow,from=idF,"\id \times\lambda"]
    \cat{1}\times\cat{C'} \ar[swap]{ddrr}{\cong} \ar[swap]{ur}{\id \times L} \ar[swap]{rr}{e\times \id } 
    \ar[bend left=10,""{name=IDLEID,pos=0.4},phantom]{urrr} \ar[bend right=10,""{name=EIDLL,pos=0.6},phantom]{urrr}
     \ar[Rightarrow,from=IDLEID,to=EIDLL,"\eta_L\times\id "]
    \ar[bend left=5,""{name=ISO},phantom]{ddrr} \ar[bend left=70,""{name=EIDT,pos=0.65},phantom]{ddrr}
     \ar[Rightarrow,from=EIDT,to=ISO,"\ell"]
    &&  \cat{C'}\times\cat{C'} \ar{dd}{\otimes} \ar{ur}{L\times L}  
     \ar[bend left=60,""{name=LLT},phantom]{dr} \ar[bend right=60,""{name=TL},phantom]{dr} 
      \ar[Rightarrow,from=LLT,to=TL,"\mu_L"] \\
   && \, && \cat{D} \\
   && \cat{C}  \ar[swap]{r}{G} &  \cat{C'} \ar[swap]{ur}{L}
  \end{tikzcd}
 \end{equation*}
 This is the same as \eqref{prismup}, precomposed with $G$ and $\lambda$. 
 Now since $\lambda$ is a Kan extension, also $e\otimes\lambda$ is one, and so by its universal property, we can cancel the precomposition, obtaining that \eqref{prismup} is equal to \eqref{prismdown}.
 
 \item The right unit condition works analogously, with $\cat{1}\times -$ replaced by $-\times\cat{1}$ and the left unitors $\ell$ replaced by the right unitors.
 
 \item Associativity means that the composite, omitting the symbol $\times$ for brevity,
 \begin{equation}\label{cubeup}
  \begin{tikzcd}
   & \cat{C'}\, \cat{C'} \ar{dr}{\otimes} \\
   \cat{C'}\, \cat{C'}\, \cat{C'} \ar{dr}{\otimes\, \id } \ar{ur}{\id \,\otimes} \ar[swap]{dd}{L\,  L\,  L} 
    \ar[bend left=30,""{name=IDTT},phantom]{rr} \ar[bend right=25,""{name=TIDT,pos=0.51},phantom]{rr} 
     \ar[Rightarrow,from=TIDT,to=IDTT,"a"]
    \ar[bend right=20,""{name=LLLT,pos=0.55},phantom]{dddr} \ar[bend left=20,""{name=TLL,pos=0.45},phantom]{dddr} 
     \ar[Rightarrow,from=LLLT,to=TLL,"\mu_L\,\id"]
    && \cat{C'} \ar{dd}{L} \\ 
   & \cat{C'}\, \cat{C'} \ar[swap]{dd}{L\,  L} \ar{ur}{\otimes}
    \ar[bend right=60,""{name=LLT,pos=0.6},phantom]{dr} \ar[bend left=60,""{name=TL,pos=0.4},phantom]{dr} 
      \ar[Rightarrow,from=LLT,to=TL,"\mu_L"] \\
   \cat{D}\, \cat{D}\, \cat{D} \ar[swap]{dr}{\otimes\, \id } && \cat{D} \\
   & \cat{D\,  D}  \ar[swap]{ur}{\otimes}
  \end{tikzcd}
 \end{equation}
 is equal to the composite
 \begin{equation}\label{cubedown}
  \begin{tikzcd}
   & \cat{C'}\, \cat{C'} \ar{dr}{\otimes} \ar{dd}{L\,  L}
    \ar[bend right=20,""{name=LLT,pos=0.55},phantom]{dddr} \ar[bend left=20,""{name=TL,pos=0.45},phantom]{dddr} 
     \ar[Rightarrow,from=LLT,to=TL,"\mu_L",swap] \\
   \cat{C'}\, \cat{C'}\, \cat{C'} \ar{ur}{\id \,\otimes} \ar[swap]{dd}{L\,  L\,  L} 
    \ar[bend right=60,""{name=LLLT,pos=0.55},phantom]{dr} \ar[bend left=60,""{name=TLL,pos=0.45},phantom]{dr} 
      \ar[Rightarrow,from=LLLT,to=TLL,"\id\,\mu_L",pos=0.4]
    && \cat{C'} \ar{dd}{L}  \\
   & \cat{D\,  D}  \ar{dr}{\otimes} \\
   \cat{D}\, \cat{D}\, \cat{D}  \ar[swap]{ur}{\id \,\otimes} \ar[swap]{dr}{\otimes\, \id }
    \ar[bend left=30,""{name=IDTT},phantom]{rr} \ar[bend right=25,""{name=TIDT,pos=0.51},phantom]{rr} 
     \ar[Rightarrow,from=TIDT,to=IDTT,"a"]
    && \cat{D} \\
   & \cat{D\,  D} \ar[swap]{ur}{\otimes}
  \end{tikzcd}
 \end{equation}
 where $a$ denotes the associator (of both categories). We think of all these 2-cells as the six faces of a cube. In order to prove the equality, we precompose the diagram \eqref{cubeup} with $G\,  G\,  G$ and $\lambda\, \lambda\, \lambda$:
 \begin{equation*}
  \begin{tikzcd}[row sep=large,column sep=large]
   && \cat{C'}\, \cat{C'} \ar{dr}{\otimes} \\
   & \cat{C'}\, \cat{C'}\, \cat{C'} \ar{r}{\otimes\, \id } \ar{ur}{\id \,\otimes} \ar{dd}{L\,  L\,  L} 
    \ar[bend left=45,""{name=IDTT,pos=0.55},phantom]{rr} \ar[bend left=15,""{name=TIDT,pos=0.51},phantom]{rr} 
     \ar[Rightarrow,from=TIDT,to=IDTT,"a"]
    \ar[bend right=20,""{name=LLLT,pos=0.5},phantom]{dddr} \ar[bend left=50,""{name=TLL,pos=0.25},phantom]{dddr} 
     \ar[Rightarrow,from=LLLT,to=TLL,"\mu_L\,\id",swap] 
    & \cat{C'}\, \cat{C'} \ar[swap]{ddd}{L\,  L} \ar{r}{\otimes} 
    \ar[bend right=90,""{name=LLT,pos=0.7},phantom]{ddr} \ar[bend left=50,""{name=TL,pos=0.4},phantom]{ddr} 
      \ar[Rightarrow,from=LLT,to=TL,"\mu_L"]
    & \cat{C'} \ar{dd}{L} \\ 
   \cat{C}\, \cat{C}\, \cat{C} \ar[bend left=10]{ur}{G\,G\,G} \ar[bend right=10,swap,""{name=FFF,above}]{dr}{F\,F\,F} 
    \ar[bend left=20,""{name=FFF},phantom]{dr} \ar[bend left=90,""{name=GGGLLL},phantom]{dr} 
     \ar[Rightarrow,from=FFF,to=GGGLLL,"\lambda\,\lambda\,\lambda"] \\
   &\cat{D}\, \cat{D}\, \cat{D}  \ar[swap]{dr}{\otimes\, \id } && \cat{D} \\
   && \cat{D\,  D} \ar[swap]{ur}{\otimes}
  \end{tikzcd}
 \end{equation*}
 Using the fact that \eqref{cylinderup} is equal to \eqref{cylinderdown} on the first two components, and that the identity commutes with $\lambda$ on the third component, this is equal to
 \begin{equation*}
  \begin{tikzcd}[row sep=large,column sep=large]
   && \cat{C'}\, \cat{C'} \ar{dr}{\otimes} \\
   & \cat{C'}\, \cat{C'}\, \cat{C'} \ar{r}{\otimes\, \id } \ar{ur}{\id \,\otimes}  
    \ar[bend left=45,""{name=IDTT,pos=0.55},phantom]{rr} \ar[bend left=15,""{name=TIDT,pos=0.51},phantom]{rr} 
     \ar[Rightarrow,from=TIDT,to=IDTT,"a"]
    &  \cat{C'}\,\cat{C'} \ar[near start]{ddd}{L\,  L} \ar{r}{\otimes} 
    \ar[bend right=90,""{name=LLT,pos=0.7},phantom]{ddr} \ar[bend left=50,""{name=TL,pos=0.4},phantom]{ddr} 
      \ar[Rightarrow,from=LLT,to=TL,"\mu_L"]
    & \cat{C'} \ar{dd}{L} \\ 
   \cat{C}\, \cat{C}\, \cat{C} \ar[bend left=10]{ur}{G\,G\,G} \ar[bend right=10,swap,""{name=FFF,above}]{dr}{F\,F\,F} \ar{r}{\otimes\, \id } 
    \ar[bend left=12,""{name=GGGTID},phantom]{urr} \ar[bend right=10,""{name=TIDGG,pos=0.46},phantom]{urr} 
     \ar[Rightarrow,from=TIDGG,to=GGGTID,"\mu_G^{-1}\,\id"]
    \ar[bend left=30,""{name=TIDFF,pos=0.38},phantom]{ddrr} \ar[bend left=10,""{name=FFFTID,pos=0.45},phantom]{ddrr} 
     \ar[Rightarrow,from=FFFTID,to=TIDFF,"\mu_F\,\id"]
    & \cat{C}\, \cat{C} \ar{ur}{G\,G} \ar{ddr}{F\,F} 
    \ar[bend left=30,""{name=FF,pos=0.4},phantom]{ddr} \ar[bend left=90,""{name=GGLL,pos=0.35},phantom]{ddr} 
     \ar[Rightarrow,from=FF,to=GGLL,"\lambda\,\lambda"] \\
   &\cat{D}\, \cat{D}\, \cat{D}  \ar[swap]{dr}{\otimes\, \id } && \cat{D} \\
   && \cat{D\,  D} \ar[swap]{ur}{\otimes}
  \end{tikzcd}
 \end{equation*}
 Using again  that \eqref{cylinderup} is equal to \eqref{cylinderdown}, we get
 \begin{equation*}
  \begin{tikzcd}[row sep=large,column sep=large]
   && \cat{C'}\, \cat{C'} \ar{dr}{\otimes} \\
   & \cat{C'}\, \cat{C'}\, \cat{C'} \ar{r}{\otimes\, \id } \ar{ur}{\id \,\otimes}  
    \ar[bend left=45,""{name=IDTT,pos=0.55},phantom]{rr} \ar[bend left=15,""{name=TIDT,pos=0.51},phantom]{rr} 
     \ar[Rightarrow,from=TIDT,to=IDTT,"a"]
    & \cat{C'}\,\cat{C'}  \ar{r}{\otimes} & |[alias=C]| \cat{C'} \ar{dd}{L} \\ 
   \cat{C}\, \cat{C}\, \cat{C} \ar[bend left=10]{ur}{G\,G\,G} \ar[bend right=10,swap,""{name=FFF,above}]{dr}{F\,F\,F} \ar{r}{\otimes\, \id } 
    \ar[bend left=12,""{name=GGGTID},phantom]{urr} \ar[bend right=10,""{name=TIDGG,pos=0.46},phantom]{urr} 
     \ar[Rightarrow,from=TIDGG,to=GGGTID,"\mu_G^{-1}\,\id"]
    \ar[bend left=30,""{name=TIDFF,pos=0.38},phantom]{ddrr} \ar[bend left=10,""{name=FFFTID,pos=0.45},phantom]{ddrr} 
     \ar[Rightarrow,from=FFFTID,to=TIDFF,"\mu_F\,\id"]
    & \cat{C}\, \cat{C} \ar{r}{\otimes}  \ar{ur}{G\,G} \ar[pos=0.7]{ddr}{F\,F} 
    \ar[bend left=12,""{name=GGTID},phantom]{urr} \ar[bend right=10,""{name=TIDG,pos=0.46},phantom]{urr} 
     \ar[Rightarrow,from=TIDG,to=GGTID,"\mu_G^{-1}",swap]
    \ar[bend right=50,""{name=FFT,pos=0.63},phantom]{drr} \ar[bend left=10,""{name=TF,pos=0.52},phantom]{drr} 
     \ar[Rightarrow,from=FFT,to=TF,"\mu_F",swap]
    & \cat{C}  \ar{ur}{G} \ar[swap]{dr}{F} 
    \ar[bend left=30,""{name=F},phantom]{dr} \ar[bend left=90,""{name=GL,pos=0.45},phantom]{dr} 
     \ar[Rightarrow,from=F,to=GL,"\lambda"] \\
   &\cat{D}\, \cat{D}\, \cat{D} \ar[swap]{dr}{\otimes\, \id } &  & \cat{D} \\
   && \cat{D\,  D}  \ar[swap]{ur}{\otimes}
  \end{tikzcd}
 \end{equation*}
 Using the associativity of $\mu_F$ and $\mu_G$, we get
 \begin{equation*}
  \begin{tikzcd}[row sep=large,column sep=large]
   && \cat{C'}\, \cat{C'} \ar{dr}{\otimes} \\
   & \cat{C'}\, \cat{C'}\, \cat{C'} \ar{ur}{\id \,\otimes}  &  & \cat{C'} \ar{dd}{L} \\ 
   \cat{C}\, \cat{C}\, \cat{C} \ar{r}{\id \,\otimes} \ar[bend left=10]{ur}{G\,G\,G} \ar[bend right=10,swap,""{name=FFF,above}]{dr}{F\,F\,F} 
    \ar[bend right=10,""{name=GGGTID,pos=0.45},phantom]{uurr} \ar[bend right=30,""{name=TIDGG,pos=0.38},phantom]{uurr} 
     \ar[Rightarrow,from=TIDGG,to=GGGTID,"\id\,\mu_G^{-1}"]
    \ar[bend left=14,""{name=TIDFF,pos=0.49},phantom]{drr} \ar[bend right=9,""{name=FFFTID,pos=0.54},phantom]{drr} 
     \ar[Rightarrow,from=FFFTID,to=TIDFF,"\id\,\mu_F"]
    & \cat{C}\,\cat{C}  \ar{dr}{F\,F} \ar[swap]{uur}{G\,G} \ar{r}{\otimes} 
     \ar[bend right=5, ""{name=TG,pos=0.53},phantom]{urr} \ar[bend left=55, ""{name=GGT,pos=0.63},phantom]{urr} 
      \ar[Rightarrow,from=TG,to=GGT,"\mu_G^{-1}",swap]
     \ar[bend left=14,""{name=TIDF,pos=0.5},phantom]{drr} \ar[bend right=9,""{name=FFTID,pos=0.54},phantom]{drr} 
      \ar[Rightarrow,from=FFTID,to=TIDF,"\mu_F",swap]
     & \cat{C} \ar{ur}{G} \ar[swap]{dr}{F}
    \ar[bend left=30,""{name=F},phantom]{dr} \ar[bend left=90,""{name=GL,pos=0.45},phantom]{dr} 
     \ar[Rightarrow,from=F,to=GL,"\lambda"] \\
   &\cat{D}\, \cat{D}\, \cat{D} \ar{r}{\id \,\otimes} \ar[swap]{dr}{\otimes\, \id } 
    \ar[bend right=12,""{name=IDTT,pos=0.52},phantom]{rr} \ar[bend right=42,""{name=TIDT,pos=0.55},phantom]{rr} 
     \ar[Rightarrow,from=TIDT,to=IDTT,"a"]
    &  \cat{D}\,\cat{D}  \ar{r}{\otimes} & \cat{D} \\
   && \cat{D\,  D}  \ar[swap]{ur}{\otimes}
  \end{tikzcd}
 \end{equation*}
 Using again  that \eqref{cylinderup} is equal to \eqref{cylinderdown},
 \begin{equation*}
  \begin{tikzcd}[row sep=large,column sep=large]
   && |[alias=CC]| \cat{C'}\, \cat{C'} \ar{dr}{\otimes} \ar{ddd}{L\,  L}  
    \ar[bend right=50,""{name=LLT,pos=0.72},phantom]{dddr} \ar[bend left=25,""{name=TL,pos=0.5},phantom]{dddr} 
     \ar[Rightarrow,from=LLT,to=TL,"\mu_L"] \\
   & \cat{C'}\, \cat{C'}\, \cat{C'} \ar{ur}{\id \,\otimes}  && \cat{C'} \ar{dd}{L} \\ 
   \cat{C}\, \cat{C}\, \cat{C} \ar{r}{\id \,\otimes} \ar[bend left=10]{ur}{G\,G\,G} \ar[bend right=10,swap,""{name=FFF,above}]{dr}{F\,F\,F} 
    \ar[bend right=10,""{name=GGGTID,pos=0.45},phantom]{uurr} \ar[bend right=30,""{name=TIDGG,pos=0.38},phantom]{uurr} 
     \ar[Rightarrow,from=TIDGG,to=GGGTID,"\id\,\mu_G^{-1}"]
    \ar[bend left=14,""{name=TIDFF,pos=0.49},phantom]{drr} \ar[bend right=9,""{name=FFFTID,pos=0.54},phantom]{drr} 
     \ar[Rightarrow,from=FFFTID,to=TIDFF,"\id\,\mu_F"]
    & \cat{C}\,\cat{C} \ar{dr}{F\,F} \ar[swap]{uur}{G\,G} 
    \ar[bend left=30,""{name=FF,pos=0.4},phantom]{dr} \ar[bend left=90,""{name=GGLL,pos=0.36},phantom]{drr} 
     \ar[Rightarrow,from=FF,to=GGLL,"\lambda\,\lambda"] \\
   &\cat{D}\, \cat{D}\, \cat{D} \ar{r}{\id \,\otimes} \ar[swap]{dr}{\otimes\, \id } 
    \ar[bend right=12,""{name=IDTT,pos=0.52},phantom]{rr} \ar[bend right=42,""{name=TIDT,pos=0.55},phantom]{rr} 
     \ar[Rightarrow,from=TIDT,to=IDTT,"a"]
    &  \cat{D}\,\cat{D}  \ar{r}{\otimes} & \cat{D} \\
   && \cat{D\,  D}  \ar[swap]{ur}{\otimes}
  \end{tikzcd}
 \end{equation*}
 Using the fact that \eqref{cylinderup} is equal to \eqref{cylinderdown} on the last two components, and that the identity commutes with $\lambda$ on the first component, this becomes
 
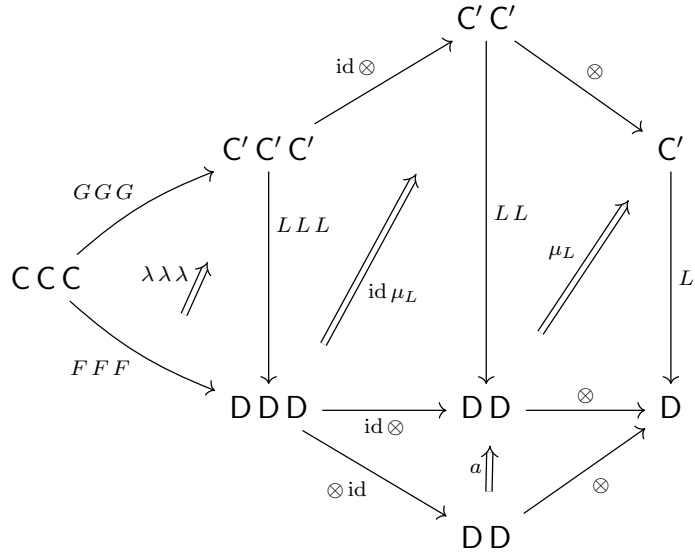
\begin{equation*}
  \begin{tikzcd}[row sep=large,column sep=large]
   && \cat{C'}\, \cat{C'} \ar{dr}{\otimes} \ar{ddd}{L\,  L} 
    \ar[bend right=50,""{name=LLT,pos=0.72},phantom]{dddr} \ar[bend left=25,""{name=TL,pos=0.5},phantom]{dddr} 
     \ar[Rightarrow,from=LLT,to=TL,"\mu_L"] \\
   &\cat{C'}\, \cat{C'}\, \cat{C'} \ar{ur}{\id \,\otimes} \ar[near start]{dd}{L\,  L\,  L} 
    \ar[bend right=50,""{name=LLLT,pos=0.58},phantom]{ddr} \ar[bend left=60, ""{name=TLL,pos=0.28}, phantom]{ddr} 
      \ar[Rightarrow,from=LLLT,to=TLL,"\id\,\mu_L",pos=0.4,swap]
    && \cat{C'} \ar{dd}{L} \\
   \cat{C}\, \cat{C}\, \cat{C} \ar[bend left=10]{ur}{G\,G\,G} \ar[bend right=10,swap,""{name=FFF,above}]{dr}{F\,F\,F}
    \ar[bend left=20,""{name=FFF},phantom]{dr} \ar[bend left=90,""{name=GGGLLL},phantom]{dr} 
     \ar[Rightarrow,from=FFF,to=GGGLLL,"\lambda\,\lambda\,\lambda"]  \\
   &\cat{D}\, \cat{D}\, \cat{D}  \ar[swap]{r}{\id \,\otimes} \ar[swap]{dr}{\otimes\, \id } 
    \ar[bend right=12,""{name=IDTT,pos=0.52},phantom]{rr} \ar[bend right=42,""{name=TIDT,pos=0.55},phantom]{rr} 
     \ar[Rightarrow,from=TIDT,to=IDTT,"a"]
    & \cat{D\,  D}  \ar{r}{\otimes} & \cat{D} \\
   && \cat{D\,  D}  \ar[swap]{ur}{\otimes}   
  \end{tikzcd}
 \end{equation*}
 which is exactly \eqref{cubedown}, precomposed with $G\,G\,G$ and $\lambda\,\lambda\,\lambda$.
 By the universal property of Kan extensions of $(\lambda\otimes\lambda)\otimes\lambda$, the precomposition with $\lambda\,\lambda\,\lambda$ and $G\,G\,G$ can be removed, obtaining that \eqref{cubeup} is equal to \eqref{cubedown}.
\end{itemize}

We now prove that this structure also makes $L$ into the left Kan extension in $\cat{MonCat}$. Given a lax monoidal functor $X:\cat{C'}\to\cat{D}$ and a monoidal transformation $\chi:F\Rightarrow X G$, we can apply the Kan extension property in $\cat{Cat}$, so that there exists a unique $u:L\Rightarrow X$ such that
\begin{equation}\label{uniu}
  \begin{tikzcd}
   \cat{C} \ar[swap]{ddr}{G} \ar{rr}{F} 
    \ar[bend right=10,""{name=F},phantom]{rr} \ar[bend right=90,""{name=GL,pos=0.48},phantom]{rr} 
     \ar[Rightarrow,bend right=20,from=F,to=GL,"\lambda",swap]
    && \cat{D} \\
   \\
   & \cat{C'} \ar[bend left,""{name=L,below}]{uur}{L} \ar[swap,bend right,""{name=X,above}]{uur}{X}
    \ar[Rightarrow,from=L,to=X,"u"]
  \end{tikzcd}
  \quad\equiv\quad
  \begin{tikzcd}
   \cat{C} \ar[swap]{ddr}{G} \ar{rr}{F} 
    \ar[bend right=20,""{name=F},phantom]{rr} \ar[bend right=90,""{name=GX,pos=0.5},phantom]{rr} 
     \ar[Rightarrow,bend left=10,from=F,to=GX,"\chi"]
    && \cat{D} \\
   \\
   & \cat{C'}  \ar[swap,bend right]{uur}{X}
  \end{tikzcd}
 \end{equation}
 What we need to show is that this $u$ is automatically monoidal. We first prove that it respects the units,
\begin{equation}\label{unitu}
 \begin{tikzcd}
  & \cat{C'} \ar[bend left=50,""{name=X,left}]{dd}{X} \ar[bend right=20,""{name=L,pos=0.51,right}]{dd}[pos=0.8]{L}
   \ar[Rightarrow,from=L,to=X,"u"] \\
  \cat{1} \ar{ur}{e} \ar[swap]{dr}{e}
   \ar[bend left=15,""{name=ED},phantom]{dr} \ar[bend left=90,""{name=ECPL,pos=0.38},phantom]{dr} 
    \ar[Rightarrow,bend left=10,from=ED,to=ECPL,"\eta_L"] \\
  & \cat{D} 
 \end{tikzcd}
 \quad \equiv \quad
 \begin{tikzcd}
  & \cat{C'} \ar[bend left=50,""{name=ECPX,below,pos=0.}]{dd}{X} \\
  \cat{1} \ar{ur}{e} \ar[swap]{dr}{e}
   \ar[bend left=15,""{name=ED},phantom]{dr} 
    \ar[Rightarrow,bend right=20,from=ED,to=ECPX,"\eta_X",swap] \\
  & \cat{D} 
 \end{tikzcd}
\end{equation}
To obtain this, we use that $\lambda$ respects units as per \eqref{unitlambda}, and similarly $\chi$. Since $\eta_G$ is an isomorphism,~\eqref{unitu} follows if we can prove it after postcomposing with $\eta_G$,
 \begin{align*}
  &\begin{tikzcd}[ampersand replacement=\&, column sep=large, row sep=large]
   \& \cat{C} \ar{dr}{G} \\
   \cat{1} \ar{rr}{e} \ar{ur}{e} \ar[swap,""{name=un,above}]{dr}{e}
    \ar[bend left=60,""{name=ECG},phantom]{rr} \ar[bend left=18,""{name=ECP,pos=0.51},phantom]{rr} 
     \ar[Rightarrow,from=ECP,to=ECG,"\eta_G"{left},"\cong"{right}]
    \ar[bend left=15,""{name=ED},phantom]{dr} \ar[bend left=90,""{name=ECPL,pos=0.68},phantom]{dr} 
     \ar[Rightarrow,bend left=10,from=ED,to=ECPL,"\eta_L",swap]
    \& \& \cat{C'} \ar[bend left,""{name=X,above}]{dl}{X} \ar[bend right,swap,""{name=L,below}]{dl}{L}
     \ar[Rightarrow,from=L,to=X,"u"] \\
   \& \cat{D}
  \end{tikzcd} 
  \quad \equiv \quad
  \begin{tikzcd}[ampersand replacement=\&, column sep=large, row sep=large]
   \& \cat{C} \ar{dr}{G} \ar[swap]{dd}{F} 
    \ar[bend left=10,""{name=F},phantom]{dd} \ar[bend left=90,""{name=GL,pos=0.48},phantom]{dd} 
     \ar[Rightarrow,bend left=20,from=F,to=GL,"\lambda"] \\
   \cat{1} \ar{ur}{e} \ar[swap,""{name=un,above}]{dr}{e} 
    \ar[bend left=15,""{name=ED},phantom]{dr} \ar[bend left=90,""{name=ECF,pos=0.42},phantom]{dr} 
     \ar[Rightarrow,from=ED,to=ECF,"\eta_F"]
    \& \& \cat{C'} \ar[bend left,""{name=X,above}]{dl}{X} \ar[bend right,swap,""{name=L,below}]{dl}[pos=0.6]{L}
     \ar[Rightarrow,from=L,to=X,"u"] \\
   \& \cat{D}
  \end{tikzcd} \\
  \equiv\quad &\begin{tikzcd}[ampersand replacement=\&, column sep=large, row sep=large]
   \& \cat{C} \ar{dr}{G} \ar[swap]{dd}{F} 
    \ar[bend left=20,""{name=F},phantom]{dd} \ar[bend left=90,""{name=GX,pos=0.5},phantom]{dd} 
     \ar[Rightarrow,bend right=10,from=F,to=GX,"\chi"] \\
   \cat{1} \ar{ur}{e} \ar[swap,""{name=un,above}]{dr}{e} 
    \ar[bend left=15,""{name=ED},phantom]{dr} \ar[bend left=90,""{name=ECF,pos=0.42},phantom]{dr} 
     \ar[Rightarrow,from=ED,to=ECF,"\eta_F"]
    \& \& \cat{C'} \ar[bend left,""{name=X,below}]{dl}{X} \\
   \& \cat{D}
  \end{tikzcd}
  \quad \equiv \quad
  \begin{tikzcd}[ampersand replacement=\&, column sep=large, row sep=large]
   \& \cat{C} \ar{dr}{G} \\
   \cat{1} \ar[""{name=ECPX,pos=0.9,below}]{rr}{e} \ar{ur}{e} \ar[swap,""{name=un,above}]{dr}{e} 
    \ar[bend left=60,""{name=ECG},phantom]{rr} \ar[bend left=18,""{name=ECP,pos=0.51},phantom]{rr} 
     \ar[Rightarrow,from=ECP,to=ECG,"\eta_G"{left},"\cong"{right}]
    \ar[bend left=15,""{name=ED},phantom]{dr} 
     \ar[Rightarrow,bend right=15,from=ED,to=ECPX,"\eta_X",swap]
    \&  \& \cat{C'} \ar[bend left]{dl}{X} \\
   \& \cat{D}
  \end{tikzcd}
 \end{align*}
which proves the claim.

Proving compatibility with the multiplication
\begin{equation}\label{multu}
 \begin{tikzcd}
  \cat{C'}\times\cat{C'} \ar[bend right=50,swap,""{name=LL,right}]{dd}{L\times L} \ar[bend left=50,""{name=XX,left}]{dd}{X\times X} \ar{rr}{\otimes}
   \ar[bend right=30,""{name=XXT,pos=0.55},phantom]{ddrr} \ar[bend left=70,""{name=TX,pos=0.65},phantom]{ddrr} 
    \ar[Rightarrow, bend right=20, from=XXT,to=TX, "\mu_X",swap]
   && \cat{C'} \ar[bend left=50]{dd}{X} \\
  \\
  \cat{D}\times\cat{D} \ar[swap]{rr}{\otimes} && \cat{D}
  \ar[Rightarrow,from=LL,to=XX,"u\times u"]
 \end{tikzcd}
 \quad \equiv \quad
 \begin{tikzcd}
  \cat{C'}\times\cat{C'} \ar[bend right=50,swap]{dd}{L\times L} \ar{rr}{\otimes}
   \ar[bend right=60,""{name=LLT,pos=0.35},phantom]{ddrr} \ar[bend left=30,""{name=TL,pos=0.4},phantom]{ddrr} 
    \ar[Rightarrow, bend left=20, from=LLT,to=TL, "\mu_L"]
   && \cat{C'} \ar[bend right=50,swap,""{name=L,right}]{dd}{L} \ar[bend left=50,""{name=X,left}]{dd}{X} 
   \ar[Rightarrow,from=L,to=X,"u"] \\
  \\
  \cat{D}\times\cat{D} \ar[swap]{rr}{\otimes} && \cat{D}
 \end{tikzcd}
\end{equation}
 works similarly, but is a bit trickier. We use compatibility of $\lambda$ with the multiplication \eqref{multlambda},
 and similarly for $\chi$, in order to compute
 \begin{align*}
  &\begin{tikzcd}[ampersand replacement=\&, row sep=large]
   \cat{C}\times\cat{C} \ar[bend left,swap]{dr}{G\times G} \ar[swap,bend right]{dd}{F\times F} \ar{rr}{\otimes} 
    \ar[bend right=5,""{name=GGT,pos=0.58},phantom]{drrr} \ar[bend left=25,""{name=TG,pos=0.7},phantom]{drrr} 
     \ar[Rightarrow, bend right=20,from=GGT,to=TG,"\mu_G","\cong"{swap}]
    \ar[bend right=10,""{name=FF,pos=0.45},phantom]{dd} \ar[bend left=60,""{name=GGLL,pos=0.42},phantom]{dd} 
     \ar[Rightarrow,bend left=20,from=FF,to=GGLL,"\lambda\times\lambda"]
    \&\& \cat{C} \ar[bend left]{dr}{G} \\
   \& \cat{C'}\times\cat{C'} \ar[swap,bend right,""{name=LL,right}]{dl}{L\times L} \ar[bend left, pos=0.15,""{name=XX,left,pos=0.5}]{dl}{X\times X}
    \ar[Rightarrow,from=LL,to=XX,"u\times u"]
   \ar[bend right=45,""{name=XXT,pos=0.45},phantom]{dr} \ar[bend left=90,""{name=TX,pos=0.75},phantom]{dr} 
    \ar[Rightarrow, bend right=15, from=XXT,to=TX, "\mu_X"]
   \ar{rr}{\otimes} \&\& \cat{C'} \ar[bend left,""{name=X,below}]{dl}{X} \\
   \cat{D}\times\cat{D}  \ar[swap]{rr}{\otimes} \&\& \cat{D}
  \end{tikzcd} \\[4pt]
  \equiv \quad
  &\begin{tikzcd}[ampersand replacement=\&, row sep=large]
   \cat{C}\times\cat{C} \ar[bend left,swap]{dr}{G\times G} \ar[swap,bend right]{dd}{F\times F} \ar{rr}{\otimes}   
    \ar[bend right=5,""{name=GGT,pos=0.58},phantom]{drrr} \ar[bend left=25,""{name=TG,pos=0.7},phantom]{drrr} 
     \ar[Rightarrow, bend right=20,from=GGT,to=TG,"\mu_G","\cong"{swap}]
    \ar[bend right=10,""{name=FF},phantom]{dd} \ar[bend left=65,""{name=GGXX,pos=0.5},phantom]{dd} 
     \ar[Rightarrow,bend right=10,from=FF,to=GGXX,"\chi\times\chi"]
    \&\& \cat{C} \ar[bend left]{dr}{G} \\
   \&  \cat{C'}\times\cat{C'} \ar[bend left, pos=0.15]{dl}{X\times X} \ar{rr}{\otimes} 
    \ar[bend right=45,""{name=XXT,pos=0.45},phantom]{dr} \ar[bend left=90,""{name=TX,pos=0.75},phantom]{dr} 
     \ar[Rightarrow, bend right=15, from=XXT,to=TX, "\mu_X"]
    \&\& \cat{C'} \ar[bend left,""{name=X,below}]{dl}{X}  \\
   \cat{D}\times\cat{D}  \ar[swap]{rr}{\otimes} \&\& \cat{D}
  \end{tikzcd}\\[4pt]
   \equiv\quad 
   &\begin{tikzcd}[ampersand replacement=\&, row sep=large]
   \cat{C}\times\cat{C} \ar[swap,bend right]{dd}{F\times F} \ar{rr}{\otimes} 
    \ar[bend right=60,""{name=FFT,pos=0.4},phantom]{ddrr} \ar[bend left=40,""{name=TF,pos=0.48},phantom]{ddrr} 
     \ar[Rightarrow,bend left=20,from=FFT,to=TF,"\mu_F",swap] 
    \&\& \cat{C} \ar[bend right,""{name=F,above},swap]{dd}{F} \ar[bend left]{dr}{G} 
    \ar[bend right=10,""{name=F},phantom]{dd} \ar[bend left=65,""{name=GX,pos=0.5},phantom]{dd} 
     \ar[Rightarrow,bend right=10,from=F,to=GX,"\chi"] \\
   \& \phantom{ \cat{C'}\times\cat{C'} } \&\&  \cat{C'} \ar[bend left,""{name=X,below}]{dl}{X} \\
   \cat{D}\times\cat{D}  \ar[swap]{rr}{\otimes} \&\& \cat{D}
  \end{tikzcd} \\[4pt]
  \equiv \quad
  &\begin{tikzcd}[ampersand replacement=\&, row sep=large]
   \cat{C}\times\cat{C} \ar[swap,bend right]{dd}{F\times F} \ar{rr}{\otimes} 
    \ar[bend right=60,""{name=FFT,pos=0.4},phantom]{ddrr} \ar[bend left=40,""{name=TF,pos=0.48},phantom]{ddrr} 
     \ar[Rightarrow,bend left=20,from=FFT,to=TF,"\mu_F",swap] 
    \&\& \cat{C} \ar[bend right,swap]{dd}{F} \ar[bend left]{dr}{G}
    \ar[bend right=10,""{name=F},phantom]{dd} \ar[bend left=60,""{name=GL,pos=0.45},phantom]{dd} 
     \ar[Rightarrow,bend left=20,from=F,to=GL,"\lambda"] \\
   \& \phantom{ \cat{C'}\times\cat{C'} } \&\& \cat{C'} \ar[bend left,""{name=X,above}]{dl}{X} \ar[bend right,swap,""{name=L,below}]{dl}[pos=0.6]{L}
    \ar[Rightarrow,from=L,to=X,"u"] \\
   \cat{D}\times\cat{D}  \ar[swap]{rr}{\otimes} \&\& \cat{D}
  \end{tikzcd}\\[4pt]
  \equiv\quad 
  &\begin{tikzcd}[ampersand replacement=\&, row sep=large]
   \cat{C}\times\cat{C} \ar[bend left,swap]{dr}{G\times G} \ar[swap,bend right]{dd}{F\times F} \ar{rr}{\otimes}
    \ar[bend right=5,""{name=GGT,pos=0.58},phantom]{drrr} \ar[bend left=25,""{name=TG,pos=0.7},phantom]{drrr} 
     \ar[Rightarrow, bend right=20,from=GGT,to=TG,"\mu_G","\cong"{swap}]
    \ar[bend right=10,""{name=FF,pos=0.45},phantom]{dd} \ar[bend left=60,""{name=GGLL,pos=0.42},phantom]{dd} 
     \ar[Rightarrow,bend left=20,from=FF,to=GGLL,"\lambda\times\lambda"]
    \&\& \cat{C} \ar[bend left]{dr}{G} \\
   \& \cat{C'}\times\cat{C'} \ar[swap,bend right,""{name=LL,above}]{dl}{L\times L} \ar{rr}{\otimes} 
    \ar[bend right=90,""{name=LLT,pos=0.33},phantom]{dr} \ar[bend left=45,""{name=TL,pos=0.52},phantom]{dr} 
     \ar[Rightarrow, bend left=15, from=LLT,to=TL, "\mu_L",swap]
    \& \& \cat{C'} \ar[bend left,""{name=X,above}]{dl}{X} \ar[bend right,swap,""{name=L,below}]{dl}[pos=0.8]{L} 
    \ar[Rightarrow,from=L,to=X,"u"] \\
   \cat{D}\times\cat{D}  \ar[swap]{rr}{\otimes} \&\& \cat{D}
  \end{tikzcd}
 \end{align*}
 Now $\lambda\otimes\lambda$ and $G\times G$ can be canceled as above. $\mu_G$ is an isomorphism, so that it can be canceled as well. We are then left with \eqref{multu}.
\end{proof}

\paragraph{Acknowledgements.}

The authors would like to thank Brendan Fong, Seerp Roald Koudenburg, Tarmo Uustalu, and Mark Weber for the interesting discussions, useful technical comments, and patient advice.

\bibliographystyle{unsrt}
\bibliography{catprob}

\begin{thebibliography}{1}

\bibitem{algkanweber}
Mark Weber.
\newblock {Algebraic {Kan} extensions along morphisms of internal algebra
  classifiers}.
\newblock {\em Tbilisi Mathematical Journal}, 9(1), 2016.
\newblock \href{https://arxiv.org/abs/1511.04911}{arXiv:1511.04911}.

\bibitem{koudenburg}
Seerp~R. Koudenburg.
\newblock {Algebraic {Kan} extensions in double categories}.
\newblock {\em Theory and Applications of Categories}, 30(5), 2015.
\newblock \href{https://arxiv.org/abs/1406.6994}{arXiv:1406.6994}.

\bibitem{ssm}
Soichiro Fujii, Shin-ya Katsumata, and Paul-André Melliès.
\newblock {Towards a formal theory of graded monads}.
\newblock In {\em {Foundations of software science and computation
  structures}}, volume 9634 of {\em {Lecture Notes in Comput. Sci.}}, page
  513–530. Springer, 2016.

\bibitem{ours_kantorovich}
Tobias Fritz and Paolo Perrone.
\newblock {A Probability Monad as the Colimit of Finite Powers}, 2017.
\newblock Submitted. \href{https://arxiv.org/abs/1712.05363}{arXiv:1712.05363}.

\bibitem{getzler}
E.~Getzler.
\newblock Operads revisited.
\newblock In {\em Algebra, Arithmetic, and Geometry}, volume~I, pages 675--698.
  Springer, 2010.

\bibitem{mt}
Paul-André Melliès and Nicolas Tabareau.
\newblock {Free models of {$T$}-algebraic theories computed as {K}an
  extensions}, 2008.
\newblock
  \href{https://hal.archives-ouvertes.fr/hal-00339331/document}{hal.archives-ouvertes.fr/hal-00339331/document}.

\bibitem{paterson}
Ross Paterson.
\newblock {Constructing applicative functors}.
\newblock In Jeremy Gibbons and Pablo Nogueira, editors, {\em {Mathematics of
  Program Construction}}, page 300–323. Springer, 2012.

\bibitem{hypervirtual}
Seerp~Roald Koudenburg.
\newblock A double-dimensional approach to formal category theory.
\newblock \href{https://arxiv.org/abs/1511.04070}{arXiv:1511.04070}.

\end{thebibliography}
\addcontentsline{toc}{section}{\bibname}

\end{document}